\documentclass{birkjour}
 \newtheorem{thm}{Theorem}[section]
 \newtheorem{cor}[thm]{Corollary}
 \newtheorem{lem}[thm]{Lemma}
 
 \theoremstyle{definition}
 
 \theoremstyle{remark}
 \newtheorem{rem}[thm]{Remark}
 
 \numberwithin{equation}{section}

\begin{document}

%
%
%
%
%
%
%
%
%

\title[Combinatorial identities associated with new families of the numbers..]
 {Combinatorial identities associated with new families of the numbers and polynomials and their approximation values}

\author[I. Kucukoglu]{Irem Kucukoglu}

\address{%
Department of Software Engineering,\\
Faculty of Engineering and Architecture,\\
Antalya Akev University,\\
Antalya, 07525, Turkey}

\email{i.kucukoglu@akev.edu.tr}

\author[Y. Simsek]{Yilmaz Simsek}
\address{Department of Mathematics,\\
	Faculty of Science University of Akdeniz,\\
	Antalya, 07058, Turkey}
\email{ysimsek@akdeniz.edu.tr}


\date{October 13, 2017}

\begin{abstract}
Recently, the numbers $Y_{n}(\lambda )$ and the polynomials $Y_{n}(x,\lambda
)$ have been introduced by the second author \cite{simsekTJM}. The purpose
of this paper is to construct higher-order of these numbers and polynomials
with their generating functions. By using these generating functions with
their functional equations and derivative equations, we derive various
identities and relations including two recurrence relations, Vandermonde
type convolution formula, combinatorial sums, the Bernstein basis functions, and also some well known
families of special numbers and their interpolation functions such as the
Apostol--Bernoulli numbers, the Apostol--Euler numbers, the Stirling numbers
of the first kind, and the zeta type function. Finally, by using Stirling's
approximation for factorials, we investigate some approximation values of the
special case of the numbers $Y_{n}\left( \lambda \right) $.
\end{abstract}

\maketitle

\section{Introduction}

The first motivation of the present paper is to define higher-order version
of the numbers $Y_{n}\left( \lambda \right) $ and polynomials $Y_{n}\left(
x;\lambda \right) $ with their generating functions which are recently
constructed by the second author \cite{simsekTJM}. The second motivation is
to derive some functional and derivative equaitons for providing new
relations and identities including these numbers and polynomials,
Vandermonde type convolution formulas, the Apostol-type numbers, the
Stirling numbers of the first kind, the Bernstein basis functions, and also
combinatorial sums which are of many applications in not only mathematics, but also aother science (\textit{cf}. \cite{Drajovic}, \cite{EgorychevSJM}, \cite{Gould}, \cite{KoshyBOOK}, \cite{SimsekMMAS2015}, \cite{SrivatavaChoi}, \cite{Sury}).

On the other hand, it is well known that the approximation theory is of
vital importance in order to derive inequalities and apply them to other
related multi-disciplinary areas. Therefore, the third motivation is to
investigate approximation values of a new numbers derived from the numbers $Y_{n}\left( \lambda \right) $ via the Stirling's approximation formula and
approximation value of the Catalan numbers.

In order to prove the results of the present paper, we first have to give
the following definitions and relations associated with some special numbers
and polynomials, and their generating functions:

The $k$-th order Apostol--Bernoulli numbers $\mathcal{B}_{n}^{\left( k\right)
}\left( \lambda \right) $ are defined by%
\begin{equation}
\left( \frac{t}{\lambda e^{t}-1}\right) ^{k}=\sum_{n=0}^{\infty }\mathcal{B}%
_{n}^{\left( k\right) }\left( \lambda \right) \frac{t^{n}}{n!}
\label{Apostol-B-hig}
\end{equation}%
where $\left\vert t\right\vert <2\pi $ when $\lambda =1$ and $\left\vert
t\right\vert <\left\vert \log \left( \lambda \right) \right\vert $ when $%
\lambda \neq 1$. Some special cases of these numbers are given as follows:

When $k=1$, (\ref{Apostol-B-hig}) reduces to the Apostol--Bernoulli numbers:%
\[
\mathcal{B}_{n}\left( \lambda \right) =\mathcal{B}_{n}^{\left( 1\right)
}\left( \lambda \right) . 
\]%
When $\lambda =1$, the numbers $\mathcal{B}_{n}\left( \lambda \right) $
reduce to the classical Bernoulli numbers%
\[
B_{n}=\mathcal{B}_{n}\left( 1\right) 
\]%
(\textit{cf}. \cite{apostol}-\cite{SrivatavaChoi}; see also the references
cited therein).

It is well known that the $k$-th order Apostol--Bernoulli numbers are
interpolated by the zeta type function $\zeta \left( \lambda ,s,k\right)$ at negative integers with the following relation:
\begin{equation}
\zeta \left( \lambda ,-m,k\right) =\sum_{n=0}^{\infty }\left( 
\begin{array}{c}
n+k-1 \\ 
n%
\end{array}%
\right) \lambda ^{n}n^{m}=-\frac{\mathcal{B}_{m+1}^{\left( k\right) }\left(
	\lambda \right) }{m+1}  \label{Apostol-Int}
\end{equation}%
(\textit{cf}. \cite{apostol}, \cite{srivas18}, \cite{SrivatavaChoi}; see
also the references cited therein).

The $k$-th order Apostol--Euler numbers $\mathcal{E}_{n}^{\left( k\right)
}\left( \lambda \right) $ are defined by means of the following generating
function:%
\begin{equation}
\left( \frac{2}{\lambda e^{t}+1}\right) ^{k}=\sum_{n=0}^{\infty }\mathcal{E}%
_{n}^{\left( k\right) }\left( \lambda \right) \frac{t^{n}}{n!}
\label{Apostol-E-hig}
\end{equation}%
where $\left\vert t\right\vert <\left\vert \log \left( -\lambda \right)
\right\vert $. Obviously,\ in the special case when $k=1$, (\ref%
{Apostol-E-hig}) reduces to the Apostol--Euler numbers:%
\[
\mathcal{E}_{n}\left( \lambda \right) =\mathcal{E}_{n}^{\left( 1\right)
}\left( \lambda \right).
\]%
For $\lambda =1$, the numbers $\mathcal{E}_{n}\left( \lambda \right) $
reduce to the classical Euler numbers%
\[
E_{n}=\mathcal{E}_{n}\left( 1\right) 
\]%
(\textit{cf}. \cite{apostol}-\cite{SrivatavaChoi}; see also the references
cited therein).

The Stirling numbers of the first kind $S_{1}\left( n,k\right) $ are defined
by:%
\begin{equation*}
\left( x\right) _{n}=\sum_{k=0}^{n}S_{1}\left( n,k\right) x^{k}
\end{equation*}%
where ${\left( x\right) }_{n}=x\left( x-1\right) \left( x-2\right) \ldots
\left( x-n+1\right) $ and their generating function is given by:
\begin{equation}
\sum_{n=k}^{\infty }S_{1}\left( n,k\right) \frac{t^{n}}{n!}=\frac{\left[
	log\left( 1+t\right) \right] ^{k}}{k!}
\label{Gen-FirstStirling}
\end{equation}
(\textit{cf}. \cite{Jordan1950}, \cite{SrivatavaChoi}; see also the
references cited in each of these earlier works).

The Catalan numbers $C_{n}$ are defined by means of the following generating
functions (\textit{cf}. \cite[p. 121]{KoshyBOOK}):

\[
\frac{1-\sqrt{1-4t}}{2t}=\sum_{n=0}^{\infty }C_{n}t^{n} 
\]%
where $0<\left\vert t\right\vert \leq \frac{1}{4}$ and $C_{0}=1$ (\textit{cf}. 
\cite[pp. 119-120]{KoshyBOOK}). The explicit formula and the recurrence
relation for the Catalan numbers are given as follows, respectively:%
\begin{equation}
C_{n}=\frac{1}{n+1}\left( 
\begin{array}{c}
2n \\ 
n%
\end{array}%
\right)  \label{Expl-Catalan}
\end{equation}%
where $n\geq 0$, and%
\begin{equation}
\frac{C_{n}}{C_{n-1}}=\frac{4n-2}{n+1},  \label{CRecur}
\end{equation}%
where $n\geq 1$\ (\textit{cf}. \cite[pp. 109-110]{KoshyBOOK}).

We summarize the present paper as follows:

Motivation of Section 1 is to give definition of higher-order of the numbers 
$Y_{n}(\lambda )$ and the polynomials $Y_{n}(x,\lambda )$. In section 2, by
using generating function and their functional equations methods, some
identities and relations, including the numbers $Y_{n}^{\left( k\right)
}\left( \lambda \right) $, the Apostol-type numbers, the Stirling numbers of
the first kind, the Bernstein basis functions, and also combinatorial sums, are derived. In Section 3, some derivative formulas and recurrence formulas
for the numbers $Y_{n}^{\left( k\right) }\left( \lambda \right) $ are
computed. In Section 4,\ by using functional equations of the generating
function, some Vandermonde type convolution formulas are derived, and some
special values of these formulas related to the Catalan numbers and
combinatorial sums are investigated. In the last section, by using
Stirling's approximation and approximation of the Catalan numbers, some
approximation value of the special case of the numbers $Y_{n}^{\left(
	k\right) }\left( \lambda \right) $ are given.

\section{The numbers $Y_{n}^{\left( k\right) }\left( \protect\lambda \right) 
	$ and the polynomials $Y_{n}^{\left( k\right) }\left( x;\protect\lambda %
	\right) $}

In this section, our motivation is to define higher-order of recently
discovered families of numbers and polynomials unifying the Apostol-type
numbers and the Apostol-type polynomials.

For $k$ nonnegative integers and $\lambda $ real or complex numbers, we
define the numbers $Y_{n}^{\left( k\right) }\left( \lambda \right) $ by the
following generating function:%
\begin{equation}
\mathcal{F}\left( t,k;\lambda \right) =\left( \frac{2}{\lambda \left(
	1+\lambda t\right) -1}\right) ^{k}=\sum_{n=0}^{\infty }Y_{n}^{\left(
	k\right) }\left( \lambda \right) \frac{t^{n}}{n!}.  \label{GenFHigOrdY}
\end{equation}%
Note that, if we set $k=1$ in (\ref{GenFHigOrdY}), the functions $\mathcal{F}%
\left( t,k;\lambda \right) $ reduce to the following generating functions
for the numbers $Y_{n}\left( \lambda \right) $ defined by Simsek in \cite%
{simsekTJM}:%
\begin{equation}
F\left( t,\lambda \right) =\frac{2}{\lambda \left( 1+\lambda t\right) -1}%
=\sum\limits_{n=0}^{\infty }Y_{n}\left( \lambda \right) \frac{t^{n}}{n!}.
\label{YNumGenFunc}
\end{equation}%
That is, the numbers $Y_{n}^{\left( k\right) }\left( \lambda \right) $
denote higher-order of the numbers $Y_{n}\left( \lambda \right) $ were
defined by Simsek in \cite{simsekTJM}.

Here, we also define the polynomials $Y_{n}^{\left( k\right) }\left(
x;\lambda \right) $ by the following generating function:%
\begin{equation}
\mathcal{F}\left( t,x,k;\lambda \right) =\mathcal{F}\left( t,k;\lambda
\right) \left( 1+\lambda t\right) ^{x}=\sum_{n=0}^{\infty }Y_{n}^{\left(
	k\right) }\left( x;\lambda \right) \frac{t^{n}}{n!}.  \label{GenFHigOrdYpoly}
\end{equation}%
Note that, if we set $k=1$ in (\ref{GenFHigOrdYpoly}), the functions $%
\mathcal{F}\left( t,x,k;\lambda \right) $ reduce to the following generating
functions for the polynomials $Y_{n}\left( x;\lambda \right) $ defined by
Simsek in \cite{simsekTJM}:%
\begin{equation}
F\left( t,x,\lambda \right) =\frac{2\left( 1+\lambda t\right) ^{x}}{\lambda
	\left( 1+\lambda t\right) -1}=\sum\limits_{n=0}^{\infty }Y_{n}\left(
x;\lambda \right) \frac{t^{n}}{n!}.  \label{YPolyGenFunc}
\end{equation}%
That is, the polynomials $Y_{n}^{\left( k\right) }\left( x;\lambda \right) $
denote higher-order of the polynomials $Y_{n}\left( x;\lambda \right) $ were
defined by Simsek in \cite{simsekTJM}.

It should be note that
\begin{eqnarray*}
	Y_{n}^{\left( k\right) }\left( \lambda \right) &=&Y_{n}^{\left( k\right)
	}\left( 0;\lambda \right)
\end{eqnarray*}
and it is obvious that%
\[
Y_{n}\left(\lambda \right) =Y_{n}^{\left( 1\right)
}\left(\lambda \right),
\]
\[
Y_{n}\left(x; \lambda \right) =Y_{n}^{\left( 1\right)}\left( x;\lambda \right).
\]

From (\ref{GenFHigOrdY}), we have%
\[
\sum_{n=0}^{\infty }Y_{n}^{\left( k\right) }\left( \lambda \right) \frac{%
	t^{n}}{n!}=\frac{2^{k}}{\left( \lambda -1\right) ^{k}\left( \frac{\lambda
		^{2}}{\lambda -1}t+1\right) ^{k}} 
\]%
By using negative binomial expansion, for $\left\vert \frac{\lambda ^{2}}{%
	\lambda -1}t\right\vert <1$, in the above equation, we get%
\begin{equation}
\sum_{n=0}^{\infty }Y_{n}^{\left( k\right) }\left( \lambda \right) \frac{%
	t^{n}}{n!}=\frac{2^{k}}{\left( \lambda -1\right) ^{k}}\sum_{n=0}^{\infty
}\left( -1\right) ^{n}\left( 
\begin{array}{c}
k+n-1 \\ 
n%
\end{array}%
\right) \left( \frac{\lambda ^{2}}{\lambda -1}\right) ^{n}t^{n}.
\label{NewtonBinomWay1}
\end{equation}%
Comparing the coefficients of $t^{n}$ on both sides of the above equation
yields an explicit formula for the numbers $Y_{n}^{\left( k\right) }\left(
\lambda \right) $ by the following theorem:

\begin{thm}
	\begin{equation}
	Y_{n}^{\left( k\right) }\left( \lambda \right) =\left( -1\right) ^{n} \left( 
	\begin{array}{c}
	k+n-1 \\ 
	n%
	\end{array}
	\right) \frac{2^{k}n!\lambda ^{2n}}{\left( \lambda -1\right) ^{k+n}}.
	\label{ExplFormulaYk}
	\end{equation}
\end{thm}

For the case of $k=1$, (\ref{ExplFormulaYk}) reduces to the explicit formula
for the numbers $Y_{n}\left( \lambda \right) $:%
\[
Y_{n}\left( \lambda \right) =\left( -1\right) ^{n}\frac{2n!}{\lambda -1}%
\left( \frac{\lambda ^{2}}{\lambda -1}\right) ^{n} 
\]%
which was given in \cite{simsekTJM}.

By using the explicit formula in (\ref{ExplFormulaYk}), a few values of the
the numbers $Y_{n}^{\left( k\right) }\left( \lambda \right) $ are computed
as follows:%
\begin{eqnarray*}
	Y_{0}^{\left( 2\right) }\left( \lambda \right) &=&\frac{4}{\left( \lambda
		-1\right) ^{2}},Y_{1}^{\left( 2\right) }\left( \lambda \right) =-\frac{%
		8\lambda ^{2}}{\left( \lambda -1\right) ^{3}}, \\
	Y_{2}^{\left( 2\right) }\left( \lambda \right) &=&\frac{24\lambda ^{4}}{%
		\left( \lambda -1\right) ^{4}},Y_{3}^{\left( 2\right) }\left( \lambda
	\right) =-\frac{96\lambda ^{6}}{\left( \lambda -1\right) ^{5}},\ldots
\end{eqnarray*}%
\begin{eqnarray*}
	Y_{0}^{\left( 3\right) }\left( \lambda \right) &=&\frac{8}{\left( \lambda
		-1\right) ^{3}},Y_{1}^{\left( 3\right) }\left( \lambda \right) =-\frac{%
		24\lambda ^{2}}{\left( \lambda -1\right) ^{4}}, \\
	Y_{2}^{\left( 3\right) }\left( \lambda \right) &=&\frac{96\lambda ^{4}}{%
		\left( \lambda -1\right) ^{5}},Y_{3}^{\left( 3\right) }\left( \lambda
	\right) =-\frac{480\lambda ^{6}}{\left( \lambda -1\right) ^{6}},\ldots.
\end{eqnarray*}

It is follows from (\ref{ExplFormulaYk}) that the numbers $Y_{n}^{\left( k\right) }$ holds a recurrence relation given by the following theorem:
\begin{thm}
	Let $n$ be a positive integer and $k$ be a nonnegative integers. By setting%
	\[
	Y_{0}^{\left( k\right) }\left( \lambda \right) =\frac{2^{k}}{\left( \lambda
		-1\right) ^{k}} 
	\]%
	the following recurrence relation holds true:%
	\[
	Y_{n}^{\left( k\right) }\left( \lambda \right) =\frac{\lambda ^{2}}{%
		1-\lambda }\left( n+k-1\right) Y_{n-1}^{\left( k\right) }\left( \lambda
	\right) . 
	\]
\end{thm}

We also give another recurrence relation for computing the numbers $%
Y_{n}^{\left( k\right) }\left( \lambda \right) $ by the following theorem:

\begin{thm}
	Let%
	\[
	Y_{0}^{\left( k\right) }\left( \lambda \right) =\frac{2^{k}}{\left( \lambda
		-1\right) ^{k}}
	\]%
	and let $n$ be positive integer. Then we have%
	\begin{equation}
	\sum_{j=0}^{k}\left( -1\right) ^{k-j}\left( n\right) _{j}\left( 
	\begin{array}{c}
	k \\ 
	j%
	\end{array}%
	\right) \lambda ^{2j}\left( 1-\lambda \right) ^{k-j}Y_{n-j}^{\left( k\right)
	}\left( \lambda \right) =0.  \label{New-7}
	\end{equation}
\end{thm}

\begin{proof}
	From (\ref{GenFHigOrdY}), we have%
	\[
	2^{k}=\left( \lambda ^{2}t+\lambda -1\right) ^{k}\sum_{n=0}^{\infty
	}Y_{n}^{\left( k\right) }\left( \lambda \right) \frac{t^{n}}{n!}.
	\]%
	Using binomial theorem yields%
	\[
	2^{k}=\sum_{n=0}^{\infty }\sum_{j=0}^{k}\left( 
	\begin{array}{c}
	k \\ 
	j%
	\end{array}%
	\right) \lambda ^{2j}\left( \lambda -1\right) ^{k-j}Y_{n}^{\left( k\right)
	}\left( \lambda \right) \frac{t^{n+j}}{n!}
	\]%
	Thus,%
	\begin{eqnarray*}
		2^{k} &=&\sum_{n=0}^{\infty }\sum_{j=0}^{k}\left( n\right) _{j}\left( 
		\begin{array}{c}
			k \\ 
			j%
		\end{array}%
		\right) \lambda ^{2j}\left( \lambda -1\right) ^{k-j}Y_{n-j}^{\left( k\right)
		}\left( \lambda \right) \frac{t^{n}}{n!} \\
		&=&\sum_{n=0}^{\infty }\sum_{j=0}^{k}\left( -1\right) ^{k-j}\left( n\right)
		_{j}\left( 
		\begin{array}{c}
			k \\ 
			j%
		\end{array}%
		\right) \lambda ^{2j}\left( 1-\lambda \right) ^{k-j}Y_{n-j}^{\left( k\right)
		}\left( \lambda \right) \frac{t^{n}}{n!}.
	\end{eqnarray*}%
	From the above equation, we get the assertion of the theorem.
\end{proof}

Furthermore, we provide a relation between the numbers $%
Y_{n}^{\left( k\right) }\left( \lambda \right) $ and the polynomials $%
Y_{n}^{\left( k\right) }\left( x;\lambda \right) $ by the following
theorem:
\begin{thm}
	Let $n$ be a nonnegative integer. Then we have%
	\[
	Y_{n}^{\left( k\right) }\left( x;\lambda \right) =\sum_{j=0}^{n}\left( 
	\begin{array}{c}
	n \\ 
	j%
	\end{array}%
	\right) \lambda ^{n-j}\left( x\right) _{n-j}Y_{j}^{\left( k\right) }\left(
	\lambda \right) .
	\]
\end{thm}

\begin{proof}
	It follows from equations (\ref{GenFHigOrdY}) and (\ref{GenFHigOrdYpoly})
	that%
	\[
	\sum_{n=0}^{\infty }Y_{n}^{\left( k\right) }\left( x;\lambda \right) \frac{%
		t^{n}}{n!}=\sum_{n=0}^{\infty }\left( x\right) _{n}\lambda ^{n}\frac{t^{n}}{%
		n!}\sum_{n=0}^{\infty }Y_{n}^{\left( k\right) }\left( \lambda \right) \frac{%
		t^{n}}{n!}. 
	\]%
	Using Cauchy product rule and equalizing the coefficients of the variable $\frac{t^{n}}{n!}$ in the previous equation yields the assertion of the theorem.
\end{proof}

\section{Some identities and relations for the numbers $Y_{n}^{\left(
		k\right) }\left( \protect\lambda \right) $}

In this section, by using generating functions and their functional
equations, we derive some identities and relations including the numbers $%
Y_{n}^{\left( k\right) }\left( \lambda \right) $, the Apostol-type numbers,
the Stirling numbers of the first kind, the Bernstein basis functions, and
also combinatorial sums.

\begin{thm}
	\[
	\sum_{j=0}^{n}\left( -1\right) ^{j}Y_{j}\left( \lambda \right) Y_{n-j}\left(
	\lambda \right) =\frac{4\lambda ^{2n}\left( 1+\left( -1\right) ^{n}\right)
		\left( n+1\right) !}{\left( n+2\right) \left( \lambda -1\right) ^{n+2}}. 
	\]
\end{thm}

\begin{proof}
	\begin{equation}
	\sum_{j=0}^{n}\left( -1\right) ^{j}Y_{j}\left( \lambda \right) Y_{n-j}\left(
	\lambda \right) =\frac{4\left( -1\right) ^{n}n!}{\left( \lambda -1\right)
		^{2}}\left( \frac{\lambda ^{2}}{\lambda -1}\right) ^{n}\sum_{j=0}^{n}\left(
	-1\right) ^{j}\frac{1}{\left( 
		\begin{array}{c}
		n \\ 
		j%
		\end{array}%
		\right) }  \label{New-2}
	\end{equation}%
	According to Simsek \cite[Eq-(13)]{SimsekFilomat2016} and Gould \cite[Eq-(5.13)]{GouldVol.3}, the following identity holds true
	\[
	\sum_{j=0}^{n}\left( -1\right) ^{j}\frac{1}{\left( 
		\begin{array}{c}
		n \\ 
		j%
		\end{array}%
		\right) }=\left( 1+\left( -1\right) ^{n}\right) \frac{n+1}{n+2}. 
	\]%
	Combining the above equation with (\ref{New-2}) yields the assertion of
	theorem.
\end{proof}

\begin{thm}
	\[
	\sum_{j=0}^{n}Y_{j}\left( \lambda \right) Y_{n-j}\left( \lambda \right) =%
	\frac{\left( -1\right) ^{n}\left( n+1\right) !}{2^{n-2}\left( \lambda
		-1\right) ^{2}}\left( \frac{\lambda ^{2}}{\lambda -1}\right)
	^{n}\sum_{j=0}^{n}\frac{2^{j}}{j+1}. 
	\]
\end{thm}

\begin{proof}
	\begin{equation}
	\sum_{j=0}^{n}Y_{j}\left( \lambda \right) Y_{n-j}\left( \lambda \right) =%
	\frac{4\left( -1\right) ^{n}n!}{\left( \lambda -1\right) ^{2}}\left( \frac{%
		\lambda ^{2}}{\lambda -1}\right) ^{n}\sum_{j=0}^{n}\frac{1}{\left( 
		\begin{array}{c}
		n \\ 
		j%
		\end{array}%
		\right) }.  \label{IYNew-1}
	\end{equation}%
	According to Sury \cite{Sury}, the following identity holds true%
	\[
	\sum_{j=0}^{n}\frac{1}{\left( 
		\begin{array}{c}
		n \\ 
		j%
		\end{array}%
		\right) }=\frac{n+1}{2^{n}}\sum_{j=0}^{n}\frac{2^{j}}{j+1}. 
	\]%
	Combining the above equation with (\ref{IYNew-1}) yields assertions of the
	theorem.
\end{proof}

By setting $\lambda \in \left[ 0,1\right] $ in (\ref{New-7}) and using the
Bernstein basis functions $B_{j}^{k}\left( \lambda \right) $ given by (cf. 
\cite{LorentzBernstein}):%
\[
B_{j}^{k}\left( \lambda \right) =\left( 
\begin{array}{c}
k \\ 
j%
\end{array}%
\right) \lambda ^{j}\left( 1-\lambda \right) ^{k-j},
\]%
we derive a relation between the Bernstein basis functions and the numbers $%
Y_{n}^{\left( k\right) }\left( \lambda \right) $ are given by the following
theorem:

\begin{thm}
	Let $\lambda \in \left[ 0,1\right] $. If $n$ is a positive integer, then we
	have%
	\[
	\sum_{j=0}^{k}\left( -1\right) ^{k-j}\left( n\right) _{j}\lambda
	^{j}B_{j}^{k}\left( \lambda \right) Y_{n-j}^{\left( k\right) }\left( \lambda
	\right) =0. 
	\]
\end{thm}

\begin{thm}
	\[
	Y_{v}^{\left( k\right) }\left( \lambda \right) =\left( -1\right)
	^{k+1}2^{k}\lambda ^{v}\sum_{m=0}^{v}\frac{S_{1}\left( v,m\right) \mathcal{B}%
		_{m+1}^{\left( k\right) }\left( \lambda \right) }{m+1}. 
	\]
\end{thm}

\begin{proof}
	Replacing $1+\lambda t$ by $e^{\log (1+\lambda t)}$ in (\ref{GenFHigOrdY}),
	for $\left\vert \lambda e^{\log (1+\lambda t)}\right\vert <1$, we have%
	\begin{eqnarray*}
		\sum_{v=0}^{\infty }Y_{v}^{\left( k\right) }\left( \lambda \right) \frac{%
			t^{v}}{v!} &=&\frac{2^{k}}{\left( \lambda e^{log(1+\lambda t)}-1\right) ^{k}}
		\\
		&=&\left( -1\right) ^{k}2^{k}\sum_{n=0}^{\infty }\left( 
		\begin{array}{c}
			n+k-1 \\ 
			n%
		\end{array}%
		\right) \lambda ^{n}e^{n\log (1+\lambda t)} \\
		&=&\left( -1\right) ^{k}2^{k}\sum_{n=0}^{\infty }\left( 
		\begin{array}{c}
			n+k-1 \\ 
			n%
		\end{array}%
		\right) \lambda ^{n}\sum_{m=0}^{\infty }\frac{\left[ n\log (1+\lambda t)%
			\right] ^{m}}{m!}.
	\end{eqnarray*}%
	Combining (\ref{Gen-FirstStirling}) with the above equation yields:
		\begin{eqnarray*}
			\sum_{v=0}^{\infty }Y_{v}^{\left( k\right) }\left( \lambda \right) \frac{%
				t^{v}}{v!} &=& \left( -1\right) ^{k}2^{k}\sum_{n=0}^{\infty }\left( 
			\begin{array}{c}
				n+k-1 \\ 
				n%
			\end{array}%
			\right) \lambda ^{n}\sum_{m=0}^{\infty }n^{m}\sum_{v=0}^{\infty }S_{1}\left(
			v,m\right) \frac{(\lambda t)^{v}}{v!}\\
			&=& \left( -1\right) ^{k}2^{k}\sum_{v=0}^{\infty }\left( \sum_{m=0}^{\infty
			}S_{1}\left( v,m\right) \sum_{n=0}^{\infty }\left( 
			\begin{array}{c}
				n+k-1 \\ 
				n%
			\end{array}%
			\right) \lambda ^{n}n^{m}\right) \frac{(\lambda t)^{v}}{v!}.
		\end{eqnarray*}%
	After equalizing the coefficients of the variable $\frac{t^{v}}{v!}$ in the
	previous equation with the necessary calculations yields:%
	\begin{equation}
	Y_{v}^{\left( k\right) }\left( \lambda \right) =\left( -1\right)
	^{k}2^{k}\lambda ^{v}\sum_{m=0}^{\infty }S_{1}\left( v,m\right)
	\sum_{n=0}^{\infty }\left( 
	\begin{array}{c}
	n+k-1 \\ 
	n%
	\end{array}%
	\right) \lambda ^{n}n^{m}.  \label{New-4}
	\end{equation}%
	Combining (\ref{Apostol-Int}) with (\ref{New-4}) yields the assertion of the
	theorem.
\end{proof}

\begin{rem}
	Observe that when $k=1$, the above equation reduces to the Theorem 9 in \cite%
	{KucukogluJNT2017}.
\end{rem}

\begin{thm}
	\[
	Y_{m}^{\left( k\right) }\left( -\lambda \right) =\left( -1\right)
	^{m+k}\lambda ^{m}\sum_{n=0}^{m}\mathcal{E}_{n}^{\left( k\right) }\left(
	\lambda \right) S_{1}\left( m,n\right) . 
	\]
\end{thm}

\begin{proof}
	When we replace $1-\lambda t$ by $e^{\log (1-\lambda t)}$ in (\ref%
	{GenFHigOrdY}), for $\left\vert \lambda e^{\log (1-\lambda t)}\right\vert <1$%
	, we have%
	\begin{eqnarray*}
		\sum_{m=0}^{\infty }Y_{m}^{\left( k\right) }\left( -\lambda \right) \frac{%
			t^{m}}{m!} &=&\frac{2^{k}\left( -1\right) ^{k}}{\left( \lambda
			e^{log(1-\lambda t)}+1\right) ^{k}} \\
		&=&\left( -1\right) ^{k}\sum_{n=0}^{\infty }\mathcal{E}_{n}^{\left( k\right)
		}\left( \lambda \right) \frac{\left[ \log (1-\lambda t)\right] ^{n}}{n!}.
	\end{eqnarray*}%
	Combining (\ref{Gen-FirstStirling}) with the above equation yields:
	\begin{eqnarray*}
		\sum_{m=0}^{\infty }Y_{m}^{\left( k\right) }\left( -\lambda \right) \frac{%
			t^{m}}{m!} &=&\left( -1\right) ^{k}\sum_{m=0}^{\infty }\left( \sum_{n=0}^{m}\left(
		-\lambda \right) ^{m}\mathcal{E}_{n}^{\left( k\right) }\left( \lambda
		\right) S_{1}\left( m,n\right) \right) \frac{t^{m}}{m!}.
	\end{eqnarray*}
	After equalizing the coefficients of the variable $\frac{t^{m}}{m!}$ in the
	previous equation with the necessary calculations yields the assertion of
	the theorem.
\end{proof}

\begin{rem}
	Observe that when $k=1$, the above equation reduces to the Theorem 10 in 
	\cite{KucukogluJNT2017}.
\end{rem}

\section{Relations arising from derivatives of the functions\ $\mathcal{F}%
	\left( t,k;\protect\lambda \right) $}

In this section, we compute some derivative formulas and recurrence formulas
for the numbers $Y_{n}^{\left( k\right) }\left( \lambda \right) $.

\begin{thm}
	\[
	Y_{n+v}^{\left( k\right) }\left( \lambda \right) =\frac{\left( -1\right)
		^{v}\left( k\right) ^{\left( v\right) }\lambda {^{2v}}}{2^{v}}Y_{n}^{\left(
		k+v\right) }\left( \lambda \right) . 
	\]
	where ${\left( x\right) }^{\left(n\right)}=x\left( x+1\right) \left( x+2\right) \ldots
	\left( x+n-1\right) $.
\end{thm}

\begin{proof}
	Differentiating the functions $\mathcal{F}\left( t,k;\lambda \right) $ with
	respect to variable $t$, we obtain the following derivative formulas:%
\[
\frac{d}{dt}\left\{ \mathcal{F}\left( t,k;\lambda \right) \right\} =-\frac{k%
}{2}\lambda {^{2}}\mathcal{F}\left( t,k+1;\lambda \right) . 
\]%
Therefore, iterating the above derivation $v$ times for the variable $t$
yields the following partial differential equation:%
\[
\frac{d^{v}}{dt^{v}}\left\{ \mathcal{F}\left( t,k;\lambda \right) \right\} =%
\frac{\left( -1\right) ^{v}\left( k\right) ^{\left( v\right) }\lambda {^{2v}}%
}{2^{v}}\mathcal{F}\left( t,k+v;\lambda \right) 
\]%
Combining (\ref{GenFHigOrdY}) with the above differential equation yields
the assertion of the theorem.
\end{proof}
\begin{thm}
	\[
	\frac{d}{d\lambda }Y_{n}^{\left( k\right) }\left( \lambda \right) =-\frac{k}{%
		2}\left( 2\lambda nY_{n-1}^{\left( k+1\right) }\left( \lambda \right)
	+Y_{n}^{\left( k+1\right) }\left( \lambda \right) \right) . 
	\]
\end{thm}

\begin{proof}
	Differentiating the functions\ $\mathcal{F}\left( t,k;\lambda \right) $ with
	respect to variable $\lambda $, we obtain the following derivative formula:
\[
\frac{d}{d\lambda }\left\{ \mathcal{F}\left( t,k;\lambda \right) \right\} =-%
\frac{k}{2}\left( 2\lambda t+1\right) \mathcal{F}\left( t,k+1;\lambda
\right) . 
\]%
From (\ref{GenFHigOrdY}), we thus have%
\begin{eqnarray*}
	\sum_{n=0}^{\infty }\frac{d}{d\lambda }Y_{n}^{\left( k\right) }\left(
	\lambda \right) \frac{t^{n}}{n!} &=&-\frac{k}{2}\left( 2\lambda t+1\right)
	\sum_{n=0}^{\infty }Y_{n}^{\left( k+1\right) }\left( \lambda \right) \frac{%
		t^{n}}{n!} \\
	&=&-k\lambda t\sum_{n=0}^{\infty }Y_{n}^{\left( k+1\right) }\left( \lambda
	\right) \frac{t^{n}}{n!}-\frac{k}{2}\sum_{n=0}^{\infty }Y_{n}^{\left(
		k+1\right) }\left( \lambda \right) \frac{t^{n}}{n!} \\
	&=&-k\lambda \sum_{n=0}^{\infty }nY_{n-1}^{\left( k+1\right) }\left( \lambda
	\right) \frac{t^{n}}{n!}-\frac{k}{2}\sum_{n=0}^{\infty }Y_{n}^{\left(
		k+1\right) }\left( \lambda \right) \frac{t^{n}}{n!} \\
	&=&\sum_{n=0}^{\infty }\left( -k\lambda nY_{n-1}^{\left( k+1\right) }\left(
	\lambda \right) -\frac{k}{2}Y_{n}^{\left( k+1\right) }\left( \lambda \right)
	\right) \frac{t^{n}}{n!}.
\end{eqnarray*}%
After equalizing the coefficients of the variable $\frac{t^{n}}{n!}$ in the
previous equation with the necessary calculations yields the assertion of
the theorem.
\end{proof}

\section{Vandermonde type convolution formula arising from functional
	equations of the generating function for the numbers $Y_{n}^{\left( k\right)
	}\left( \protect\lambda \right) $}

In this section, we give some Vandermonde type convolution formulas drived
from functional equations of the function $\mathcal{F}\left( t,k;\lambda
\right) $. We also give some special values of these formulas related to the
Catalan numbers and combinatorial sums.

Let $m\in \mathbb{N}$ and $k_{1},k_{2},\ldots ,k_{m}\in \mathbb{N}$. Using (%
\ref{GenFHigOrdY}) yields the following functional equation:%
\[
\mathcal{F}\left( t,k_{1}+k_{2}+\cdots +k_{m};\lambda \right) =\mathcal{F}%
\left( t,k_{1};\lambda \right) \mathcal{F}\left( t,k_{2};\lambda \right)
\ldots \mathcal{F}\left( t,k_{m};\lambda \right) . 
\]%
By using the above functional equation, we get%
\[
\sum_{n=0}^{\infty }Y_{n}^{\left( k_{1}+k_{2}+\cdots +k_{m}\right) }\left(
\lambda \right) \frac{t^{n}}{n!}=\left( \sum_{n=0}^{\infty }Y_{n}^{\left(
	k_{1}\right) }\left( \lambda \right) \frac{t^{n}}{n!}\right) \cdots \left(
\sum_{n=0}^{\infty }Y_{n}^{\left( k_{m}\right) }\left( \lambda \right) \frac{%
	t^{n}}{n!}\right) . 
\]%
Using the Cauchy product\ in the above equation, we obtain%
\begin{eqnarray*}
	\sum_{n=0}^{\infty }Y_{n}^{\left( k_{1}+k_{2}+\cdots +k_{m}\right) }\left(
	\lambda \right) \frac{t^{n}}{n!}
	&=&\sum_{v_{1}+v_{2}+\cdots +v_{m-1}=n}\frac{Y_{v_{m-1}}^{\left(
			k_{m}\right) }\left( \lambda \right) }{\left( v_{m-1}\right) !}
	 \frac{Y_{v_{m-2}}^{\left( k_{m-1}\right) }\left( \lambda \right) }{\left(
		v_{m-2}\right) !}\cdots
	\\ 	&& \times \frac{Y_{v_{1}}^{\left( k_{1}\right) }\left( \lambda
		\right) }{v_{1}!}  \frac{Y_{n-v_{1}-v_{2}-\cdots -v_{m-1}}^{\left(
			k_{2}\right) }\left( \lambda \right) }{\left( n-v_{1}-v_{2}-\cdots
		-v_{m-1}\right) !}t^{n}.
\end{eqnarray*}%
where 
\[
\sum_{v_{1}+v_{2}+\cdots +v_{m-1}=n} 
\]%
denotes%
\[
\sum_{v_{m-1}=0}^{n}\sum_{v_{m-2}=0}^{n-v_{m-1}}\cdots
\sum_{v_{1}=0}^{n-v_{2}-v_{3}-\cdots -v_{m-1}}. 
\]%
After equalizing the coefficients of the variable $t^{n}$ in the previous
equation with the necessary calculations yields the following theorem:

\begin{thm}
	\label{Theorem-Multi}%
	Let $m\in \mathbb{N}$, $k_{1},k_{2},\ldots ,k_{m}\in \mathbb{N}$ and $n\in \mathbb{N}_0$. Then we have 
	\begin{eqnarray*}
		&&Y_{n}^{\left( k_{1}+k_{2}+\cdots +k_{m}\right) }\left( \lambda \right) \\
		&=&\sum_{v_{1}+v_{2}+\cdots +v_{m-1}=n}\mathcal{C}_{v_{1},v_{2},\ldots
			,v_{m-1}}^{n}Y_{v_{m-1}}^{\left( k_{m}\right) }\left( \lambda \right)
		Y_{v_{m-2}}^{\left( k_{m-1}\right) }\left( \lambda \right) \cdots
		Y_{v_{1}}^{\left( k_{1}\right) }\left( \lambda \right) \\
		&&\times Y_{n-v_{1}-v_{2}-\cdots -v_{m-1}}^{\left( k_{2}\right) }\left(
		\lambda \right),
	\end{eqnarray*}%
	where%
	\begin{eqnarray*}
		\mathcal{C}_{v_{1},v_{2},\ldots ,v_{m-1}}^{n} &=&\left( 
		\begin{array}{c}
			n \\ 
			v_{1},v_{2},\ldots ,n-v_{1}-\cdots -v_{m-1}%
		\end{array}%
		\right) \\
		&=&\frac{n!}{v_{1}!v_{2}!\ldots \left( n-v_{1}-\cdots -v_{m-1}\right) !}.
	\end{eqnarray*}
\end{thm}

\begin{rem}
	Substituting $m=2$ into Theorem \ref{Theorem-Multi}, we have%
	\begin{equation}
	Y_{n}^{\left( k_{1}+k_{2}\right) }\left( \lambda \right)
	=\sum_{v_{1}=0}^{n}\left( 
	\begin{array}{c}
	n \\ 
	v_{1}%
	\end{array}%
	\right) Y_{v_{1}}^{\left( k_{1}\right) }\left( \lambda \right)
	Y_{n-v_{1}}^{\left( k_{2}\right) }\left( \lambda \right) .
	\label{Calc-HigherOrderYk}
	\end{equation}
\end{rem}

\begin{rem}
	Substituting $m=3$ into Theorem \ref{Theorem-Multi}, we have%
	\begin{eqnarray*}
	Y_{n}^{\left( k_{1}+k_{2}+k_{3}\right) }\left( \lambda \right)
	&=&\sum_{v_{2}=0}^{n}\sum_{v_{1}=0}^{n-v_{2}}\left( 
	\begin{array}{c}
		n \\ 
		v_{2}%
	\end{array}%
	\right) \left( 
	\begin{array}{c}
		n-v_{2} \\ 
		v_{1}%
	\end{array}%
	\right) Y_{v_{1}}^{\left( k_{1}\right) }\left( \lambda \right)
	\\
	&& \times
	Y_{n-v_{1}-v_{2}}^{\left( k_{2}\right) }\left( \lambda \right)
	Y_{v_{2}}^{\left( k_{3}\right) }\left( \lambda \right) . 
	\end{eqnarray*}
\end{rem}

We now give the Vandermonde type convolution formula by the following
theorem with the help of the explicit formula for the numbers $Y_{n}^{\left(
	k\right) }\left( \lambda \right) $.

\begin{thm}
	\label{Theorem-MultiVanderType}Let $m\in \mathbb{N}$, $k_{1},k_{2},\ldots ,k_{m}\in \mathbb{N}$ and $n\in \mathbb{N}_0$. Then we have 
	\begin{eqnarray*}
		&&\left( 
		\begin{array}{c}
			k_{1}+k_{2}+\cdots +k_{m}+n-1 \\ 
			n%
		\end{array}%
		\right) \\
		&=&\sum_{v_{1}+v_{2}+\cdots +v_{m-1}=n}\left( 
		\begin{array}{c}
			k_{m}+v_{m-1}-1 \\ 
			v_{m-1}%
		\end{array}%
		\right) \left( 
		\begin{array}{c}
			k_{m-1}+v_{m-2}-1 \\ 
			v_{m-2}%
		\end{array}%
		\right) \cdots \\
		&&\times \left( 
		\begin{array}{c}
			k_{1}+v_{1}-1 \\ 
			v_{1}%
		\end{array}%
		\right) \left( 
		\begin{array}{c}
			k_{2}+n-v_{1}-v_{2}-\cdots -v_{m-1}-1 \\ 
			n-v_{1}-v_{2}-\cdots -v_{m-1}%
		\end{array}%
		\right) .
	\end{eqnarray*}
\end{thm}

\begin{proof}
	By combining (\ref{ExplFormulaYk}) with Theorem \ref{Theorem-Multi} yields
	the Vandermonde type convolution formula. So we omit the detail of proof.
\end{proof}

We are ready to give some special cases of Theorem \ref%
{Theorem-MultiVanderType} by the following corollaries:

Substituting $m=2$ into Theorem \ref{Theorem-MultiVanderType}, we get the
following corollary:

\begin{cor}
	Let $k_{1},k_{2}\in \mathbb{N}$ and $n\in \mathbb{N}_0$. Then we have 
	\begin{equation}
	\left( 
	\begin{array}{c}
	k_{1}+k_{2}+n-1 \\ 
	n%
	\end{array}%
	\right) =\sum_{v_{1}=0}^{n}\left( 
	\begin{array}{c}
	k_{1}+v_{1}-1 \\ 
	v_{1}%
	\end{array}%
	\right) \left( 
	\begin{array}{c}
	k_{2}+n-v_{1}-1 \\ 
	n-v_{1}%
	\end{array}%
	\right) .  \label{MultiVander2}
	\end{equation}
\end{cor}

\begin{proof}
	Combining (\ref{Calc-HigherOrderYk}) with (\ref{ExplFormulaYk}) yields%
	\begin{eqnarray}
	&&\sum_{v_{1}=0}^{n}\left( 
	\begin{array}{c}
	n \\ 
	v_{1}%
	\end{array}%
	\right) Y_{v_{1}}^{\left( k_{1}\right) }\left( \lambda \right)
	Y_{n-v_{1}}^{\left( k_{2}\right) }\left( \lambda \right)  \label{Th-2-i} \\
	&=&\left( -1\right) ^{n}n!\left( \frac{2}{\lambda -1}\right)
	^{^{k_{1}+k_{2}}}\left( 
	\begin{array}{c}
	k_{1}+k_{2}+n-1 \\ 
	n%
	\end{array}%
	\right) \left( \frac{\lambda ^{2}}{\lambda -1}\right) ^{n}.  \nonumber
	\end{eqnarray}%
	By combining (\ref{ExplFormulaYk}) with (\ref{Th-2-i}), we obtain%
	\begin{eqnarray*}
		&&\sum_{v_{1}=0}^{n}\left( 
		\begin{array}{c}
			n \\ 
			v_{1}%
		\end{array}%
		\right) \left( \left( -1\right) ^{v_{1}}\left( 
		\begin{array}{c}
			k_{1}+v_{1}-1 \\ 
			v_{1}%
		\end{array}%
		\right) \frac{2^{k_{1}}v_{1}!\lambda ^{2v_{1}}}{\left( \lambda -1\right)
			^{k_{1}+v_{1}}}\right) \\
		&&\times \left( \left( -1\right) ^{n-v_{1}}\left( 
		\begin{array}{c}
			k_{2}+n-v_{1}-1 \\ 
			n-v_{1}%
		\end{array}%
		\right) \frac{2^{k_{2}}\left( n-v_{1}\right) !\lambda ^{2\left(
				n-v_{1}\right) }}{\left( \lambda -1\right) ^{k_{2}+n-v_{1}}}\right) \\
		&=&\left( -1\right) ^{n}n!\left( \frac{2}{\lambda -1}\right)
		^{^{k_{1}+k_{2}}}\left( 
		\begin{array}{c}
			k_{1}+k_{2}+n-1 \\ 
			n%
		\end{array}%
		\right) \left( \frac{\lambda ^{2}}{\lambda -1}\right) ^{n}.
	\end{eqnarray*}%
	After the necessary calculations in the previous equation yields the desired
	result.
\end{proof}

\begin{rem}
	Substituting $k_{1}=a+1$ and $k_{2}=r+1$ into (\ref{MultiVander2}) yields
	Eq-(1.78) in \cite{GouldVOL4}, substituting $k_{1}=k_{2}=r+1$ into (\ref%
	{MultiVander2}) yields Eq-(1.79) in \cite{GouldVOL4}, and also substituting $%
	k_{1}=r+1$ and $k_{2}=n+r+1$ into (\ref{MultiVander2}) yields Eq-(1.82) in 
	\cite{GouldVOL4}.
\end{rem}

Substituting $m=3$ into Theorem \ref{Theorem-MultiVanderType}, we obtain the
following corollary:

\begin{cor}
	Let $k_{1},k_{2},k_{3}\in \mathbb{N}$ and let $n\in \mathbb{N}_0$. Then we have 
	\begin{eqnarray*}
		&&\left( 
		\begin{array}{c}
			k_{1}+k_{2}+k_{3}+n-1 \\ 
			n%
		\end{array}%
		\right) \\
		&=&\sum_{v_{2}=0}^{n}\sum_{v_{1}=0}^{n-v_{2}}\left( 
		\begin{array}{c}
			k_{1}+v_{1}-1 \\ 
			v_{1}%
		\end{array}%
		\right) \left( 
		\begin{array}{c}
			k_{2}+n-v_{1}-v_{2}-1 \\ 
			n-v_{1}-v_{2}%
		\end{array}%
		\right) \left( 
		\begin{array}{c}
			k_{3}+v_{2}-1 \\ 
			v_{2}%
		\end{array}%
		\right) .
	\end{eqnarray*}
\end{cor}

\begin{rem}
	The following well known the Chu-Vandermonde identity given as follows:%
	\[
	\left( 
	\begin{array}{c}
	x+a \\ 
	k%
	\end{array}%
	\right) =\sum_{j=0}^{k}\left( 
	\begin{array}{c}
	x \\ 
	j%
	\end{array}%
	\right) \left( 
	\begin{array}{c}
	a \\ 
	k-j%
	\end{array}%
	\right) 
	\]%
	(\textit{cf}. \cite{Gould}, \cite{Jordan1950}, \cite{SimsekMMAS2015}).
\end{rem}

We give some applications related to the Vandermonde type convolution
formula and also some combinatorial sums as follows:

Substituting $k_{1}=k_{2}=n$ into (\ref{MultiVander2}), and since 
\[
\left( 
\begin{array}{c}
3n-1 \\ 
n%
\end{array}%
\right) =\frac{2}{3}\left( 
\begin{array}{c}
3n \\ 
n%
\end{array}%
\right) , 
\]%
we obtain the following corollary:

\begin{cor}
	Let $n\in \mathbb{N}_0$. Then we have 
	\[
	\left( 
	\begin{array}{c}
	3n \\ 
	n%
	\end{array}%
	\right) =\frac{3}{2}\sum_{j=0}^{n}\left( 
	\begin{array}{c}
	n+j-1 \\ 
	j%
	\end{array}%
	\right) \left( 
	\begin{array}{c}
	2n-j-1 \\ 
	n-j%
	\end{array}%
	\right) . 
	\]%
\end{cor}

Substituting $k_{1}=n+1$ and $k_{2}=n$ into (\ref{MultiVander2}), we obtain the following corollary:

\begin{cor}
	Let $n\in \mathbb{N}_0$. Then we have 
	\begin{equation}
	\left( 
	\begin{array}{c}
	3n \\ 
	n%
	\end{array}%
	\right) =\sum_{j=0}^{n}\left( 
	\begin{array}{c}
	n+j \\ 
	j%
	\end{array}%
	\right) \left( 
	\begin{array}{c}
	2n-j-1 \\ 
	n-j%
	\end{array}%
	\right) .  \label{New-5}
	\end{equation}
\end{cor}

\begin{cor}
	Let $n\in \mathbb{N}_{0}$. Then we have%
	\begin{equation}
	C_{n}=\frac{1}{n+1}\left( 
	\begin{array}{c}
	3n \\ 
	n%
	\end{array}%
	\right) -\frac{1}{n+1}\sum_{j=0}^{n-1}\left( 
	\begin{array}{c}
	n+j \\ 
	j%
	\end{array}%
	\right) \left( 
	\begin{array}{c}
	2n-j-1 \\ 
	n-j%
	\end{array}%
	\right) .  \label{New-6}
	\end{equation}
\end{cor}

\begin{proof}
	By using (\ref{New-5}), we have%
	\[
	\left( 
	\begin{array}{c}
	3n \\ 
	n%
	\end{array}%
	\right) -\left( 
	\begin{array}{c}
	2n \\ 
	n%
	\end{array}%
	\right) =\sum_{j=0}^{n-1}\left( 
	\begin{array}{c}
	n+j \\ 
	j%
	\end{array}%
	\right) \left( 
	\begin{array}{c}
	2n-j-1 \\ 
	n-j%
	\end{array}%
	\right) .
	\]%
	Multiplying both sides of the above equation by $\frac{1}{n+1}$\ and using (%
	\ref{Expl-Catalan}) yields the desired results.
\end{proof}

In \cite[Corollary 6.6]{SimsekMMAS2015}, Simsek gave the following formulas
for the Catalan numbers:%
\[
C_{n}=\frac{2n+1}{n+1}\sum_{j=n}^{2n}\left( -1\right) ^{n+j}\left( 
\begin{array}{c}
2n \\ 
j%
\end{array}%
\right) \frac{1}{j+1}.
\]%
Combining the above equation with (\ref{New-6}), we arrive at a
combinatorial sum including binomial coefficients given by the following theorem:

\begin{thm}
	Let $n\in \mathbb{N}_0$. Then we have 
	\[
	\left( 
	\begin{array}{c}
	3n \\ 
	n%
	\end{array}%
	\right) =\sum_{j=n}^{2n}\left( -1\right) ^{n+j}\left( 
	\begin{array}{c}
	2n \\ 
	j%
	\end{array}%
	\right) \frac{2n+1}{j+1}+\sum_{j=0}^{n-1}\left( 
	\begin{array}{c}
	n+j \\ 
	j%
	\end{array}%
	\right) \left( 
	\begin{array}{c}
	2n-j-1 \\ 
	n-j%
	\end{array}%
	\right) .
	\]
\end{thm}

\section{Some applications related to the numbers $Y_{n}^{\left( n+1\right)
	}\left( \protect\lambda \right) $ and their approximation}

Inequalities are used in all multi-disciplinary areas which are
related to mathematics and its applications. Especially, inequalities have
many applications in approximation theory and related areas.

In this section, by using Stirling's approximation for factorials, we give
some approximation value of the special case of the numbers $Y_{n}^{\left(
	k\right) }\left( \lambda \right) $. The results of this section may have
potential usage in areas including inequalities and approximation theory.

Firstly, let us recall an approximate value of the Catalan numbers is given
by the following lemma:

\begin{lem}
	\textup{(\textit{cf}. \cite[p. 110]{KoshyBOOK})}
	\begin{equation}
	C_{n}\approx \frac{2^{2n}}{n\sqrt{n\pi }}.  \label{CApp}
	\end{equation}
\end{lem}

Secondly, let us recall the following Stirling's approximation formula (%
\textit{cf}. \cite{KoshyBOOK}, \cite{SrivatavaChoi}):%
\begin{equation}
n!\approx \left( \frac{n}{e}\right) ^{n}\sqrt{2\pi n},
\label{StirlingApprox}
\end{equation}%
which is used in the proof of equation (\ref{CApp}).

Now, we define the numbers $V_{n}\left( \lambda \right) $ as follows:%
\begin{equation}
V_{n}\left( \lambda \right) :=\frac{Y_{n}^{\left( n+1\right) }\left( \lambda
	\right) }{\left( n+1\right) !}.  \label{NewNumberVn}
\end{equation}%
By substituting $k=n+1$ into (\ref{ExplFormulaYk}), we obtain%
\begin{equation}
Y_{n}^{\left( n+1\right) }\left( \lambda \right) =\left( -1\right)
^{n}\left( 
\begin{array}{c}
2n \\ 
n%
\end{array}%
\right) \frac{2^{n+1}n!\lambda ^{2n}}{\left( \lambda -1\right) ^{2n+1}}.
\label{New-1}
\end{equation}%
Combining (\ref{NewNumberVn}) with (\ref{New-1}) yields a relation between
the Catalan numbers and the numbers $V_{n}\left( \lambda \right) $ given by
the following theorem:

\begin{thm}
	\begin{equation}
	V_{n}\left( \lambda \right) =\left( -1\right) ^{n}C_{n}\frac{2^{n+1}\lambda
		^{2n}}{\left( \lambda -1\right) ^{2n+1}}.  \label{THNew-1}
	\end{equation}
\end{thm}

Applying (\ref{CApp}) and (\ref{StirlingApprox}) to (\ref{THNew-1}) yields
an approximation value of the numbers $V_{n}\left( \lambda \right) $ by the
following theorem:

\begin{thm}
	Let $n$ be sufficiently large. Then we have%
	\begin{equation}
	V_{n}\left( \lambda \right) \approx \frac{2^{3n+1}\left( -\lambda
		^{2}\right) ^{n}}{\left( \lambda -1\right) ^{2n+1}n^{\frac{3}{2}}\sqrt{\pi }}%
	.  \label{New-8}
	\end{equation}
\end{thm}

As an application of the above theorem, we give the following examples:

When $n=1$ and $\lambda =-1$ in the right hand side of the (\ref{New-8}) is
as follows:%
\[
\frac{2^{3n+1}\left( -\lambda ^{2}\right) ^{n}}{\left( \lambda -1\right)
	^{2n+1}n^{\frac{3}{2}}\sqrt{\pi }}=\frac{2}{\sqrt{\pi }}\approx 1,1283
\]%
whereas $V_{1}\left( -1\right) =0,5$.

When $n=5$ and $\lambda =-1$ in the
right hand side of the (\ref{New-8}) is as follows:%
\[
\frac{2^{3n+1}\left( -\lambda ^{2}\right) ^{n}}{\left( \lambda -1\right)
	^{2n+1}n^{\frac{3}{2}}\sqrt{\pi }}=\frac{32}{5\sqrt{5\pi }}\approx 1,6148
\]%
whereas $V_{5}\left( -1\right) =1,3125$.

When $n=10$ and $\lambda =-1$ in
the right hand side of the (\ref{New-8}) is as follows:%
\[
\frac{2^{3n+1}\left( -\lambda ^{2}\right) ^{n}}{\left( \lambda -1\right)
	^{2n+1}n^{\frac{3}{2}}\sqrt{\pi }}=-\frac{2^{10}}{10\sqrt{10\pi }}\approx
-18,2694
\]%
whereas $V_{10}\left( -1\right) =-\frac{4199}{256}=-16,40234375$.

When $n=125
$ and $\lambda =-1$ in the right hand side of the (\ref{New-8}) is as
follows:%
\[
\frac{2^{3n+1}\left( -\lambda ^{2}\right) ^{n}}{\left( \lambda -1\right)
	^{2n+1}n^{\frac{3}{2}}\sqrt{\pi }}=\frac{2^{125}}{125\sqrt{125\pi }}\approx
1,7171\times 10^{34}
\]%
whereas%
\[
V_{125}\left( -1\right) =\frac{\left( 
	\begin{array}{c}
	250 \\ 
	125%
	\end{array}%
	\right) }{2^{125}126}\approx 1,7018\times 10^{34}.
\]

As a consequence, from (\ref{THNew-1}), we obtain%
\[
\frac{V_{n+1}\left( \lambda \right) }{V_{n}\left( \lambda \right) }=-2\frac{%
	C_{n+1}}{C_{n}}\left( \frac{\lambda }{\lambda -1}\right) ^{2}. 
\]%
Combining the above equation with (\ref{CRecur}), we get the following
recurrence relation for the numbers $V_{n}\left( \lambda \right) $ as
follows:%
\[
\frac{V_{n+1}\left( \lambda \right) }{V_{n}\left( \lambda \right) }=-\frac{%
	8n+4}{n+2}\left( \frac{\lambda }{\lambda -1}\right) ^{2}. 
\]%
When $n$ is sufficiently large, we arrive at the following theorem:

\begin{thm}
	\[
	V_{n+1}\left( \lambda \right) \approx -8\left( \frac{\lambda }{\lambda -1}%
	\right) ^{2}V_{n}\left( \lambda \right) . 
	\]
\end{thm}

\section{Conclusion}

Since the results of this paper include combinatorial sums, some special numbers and polynomials, some special functions, the Chu-Vandermonde type identities, and the Stirling type
approximation formulas for some special numbers, there exists potential
applications of these results to not only subfields of mathematics such as
theory of inequalities and approximation theory, but also to areas of
physics, engineering and the other multi-disciplinary areas.

\subsection*{Acknowledgment}
	The present investigation was supported by the Scientific Research Project Administration of Akdeniz University.


\begin{thebibliography}{1}
\bibitem{apostol} Apostol, T.M.: On the Lerch zeta function. Pacific J. Math. \textbf{1}, 161-167 (1951).

\bibitem{Boyadzhiev} Boyadzhiev, K.N.: Apostol--Bernoulli functions, derivative polynomials and Eulerian polynomials. arXiv:0710.1124v1.

\bibitem{Drajovic} Djordjevic, G.B., Milovanovic G.V.: Special classes of
polynomials. Leskovac: University of Nis, Faculty of Technology (2014).

\bibitem{EgorychevSJM} Egorychev, G.P., Yuzhakov, A.P.: The determination of generating functions and combinatorial sums by multidimensional residues. Sib. Math. J. (Translated from Sibirskii Matematicheskii Zhurnal \textbf{15}(5), 1049-1060 (1973)).

\bibitem{Gould} Gould, H.W.: Inverse series relations and other expansions
involving Humbert polynomials. Duke Math. J. \textbf{32}(4), 697-712 (1965).

\bibitem{GouldVOL4} Gould, H.W.: Combinatorial Identities: Table I:
Intermediate Techniques for Summing Finite Series,   http://www.math.wvu.edu/\symbol{126}gould/Vol.4.PDF

\bibitem{GouldVol.3} Gould, H.W.: Fundamentals of Series: Table III: Basic
Algebraic Techniques, http://www.math.wvu.edu/\symbol{126}gould/Vol.3.PDF

\bibitem{Jordan1950} Jordan, C.: Calculus of Finite Differences. 2nd ed. Chelsea Publishing Company, New York (1950).

\bibitem{DSkim2} Kim, D.S., Kim, T., Seo, J.: A note on Changhee numbers and
polynomials. Adv. Stud. Theor. Phys. \textbf{7}, 993-1003 (2013).

\bibitem{DSkimDaehee} Kim, D.S., Kim, T.: Daehee numbers and polynomials. Appl. Math. Sci. (Ruse) \textbf{7}, 5969-5976 (2013).

\bibitem{KoshyBOOK} Koshy, T.: Catalan numbers with applications. Oxford
University Press, Oxford, New York (2009).

\bibitem{LorentzBernstein} Lorentz G.G.: Bernstein Polynomials. Chelsea
Pub. Comp, New York (1986).

\bibitem{Luo2006} Luo, Q.-M.: Apostol--Euler polynomials of higher order and
Gaussian hypergeometric functions. Taiwanese J. Math. \textbf{10}, 917-925 (2006).

\bibitem{Luo2009} Luo, Q.-M.: The multiplication formulas for the
Apostol--Bernoulli and Apostol--Euler polynomials of higher order. Integral
Trans. Special Func. \textbf{20}, 377-391 (2009).

\bibitem{LuSrivastava} Lu, D.Q., Srivastava, H.M.: Some series identities
involving the generalized Apostol type and related polynomials. Comput. Math. Appl. \textbf{62}, 3591-3602 (2011).

\bibitem{LuoSrivastava2005} Luo, Q.M., Srivastava H.M.: Some
generalizations of the Apostol--Bernoulli and Apostol--Euler polynomials, J. Math. Anal. Appl. \textbf{308}, 290-302 (2005).

\bibitem{Luo} Luo, Q.M., Srivastava, H.M.: Some generalizations of the
Apostol--Genocchi polynomials and the Stirling numbers of the second kind.
Appl. Math. Compute. \textbf{217}, 5702-5728 (2011).

\bibitem{OzdenAMC2014} Ozden, H., Simsek, Y.: Modification and unification of
the Apostol-type numbers and polynomials and their applications. Appl. Math. Compute. \textbf{235}, 338-351 (2014).

\bibitem{SimsekFPTA} Simsek, Y.: Generating functions for generalized Stirling type numbers, array type polynomials, Eulerian type polynomials and their alications. Fixed Point Theory Appl. \textbf{87}, 343-355 (2013).

\bibitem{SimsekMMAS2015} Simsek, Y.: Analysis of the Bernstein basis
functions: an approach to combinatorial sums involving binomial coefficients
and Catalan numbers. Math. Method. Appl. Sci. \textbf{38}(14), 3007-3021 (2015).

\bibitem{SimsekFilomat2016} Simsek, Y.: Combinatorial Identities Associated
with Bernstein Type Basis Functions. Filomat \textbf{30}(7), 1683-1689 (2016).

\bibitem{simsekTJM} Simsek, Y.: Construction some new families of Apostol-type numbers and polynomials via Dirichlet character and $p$-adic integrals. Turkish J. Math. 2017, doi: 10.3906/mat-1703-114.

\bibitem{SrivastavaManocha} Srivastava, H.M., Manocha, H.L.: A Treatise on
Generating Functions. Ellis Horwood Limited Publisher, Chichester (1984).

\bibitem{srivas18} Srivastava, H.M., Kim, T., Simsek, Y.: $q$-Bernoulli numbers and
polynomials associated with multiple $q$-zeta functions and basic $L$%
-series. Russ. J. Math. Phys. \textbf{12}, 241-268 (2005).

\bibitem{SrivatavaChoi} Srivastava, H.M., Choi, J.: Zeta and $q$-zeta functions
and associated series and integrals. Elsevier Science Publishers, Amsterdam, London and New York (2012).

\bibitem{KucukogluJNT2017} Srivastava, H.M., Kucukoglu, I., Simsek, Y.: Partial differential equations for a new family of numbers and polynomials unifying the Apostol-type numbers and the Apostol-type polynomials. J. Number Theory \textbf{181}, 117-146 (2017).

\bibitem{Sury} Sury, B.: Sum of the reciprocals of the binomial coefficients.
European J. Combin. \textbf{14}, 351-353 (1993).
\end{thebibliography}
\end{document}